\documentclass[12pt]{article}
\usepackage{amsmath,amssymb}
\usepackage{graphicx}
\newtheorem{theorem}{Theorem}[section]

\newtheorem{lemma}[theorem]{Lemma}
\def\bull{\vrule height .9ex width .8ex depth -.1ex}
\newenvironment{proof}{\smallbreak \noindent {\bf Proof.~}}
              {\unskip\nobreak\hfill\hskip 2em \bull\par\medbreak}
\newenvironment{proofof}[1]{\medbreak\noindent{\bf Proof of~#1.~}}
              {\unskip\nobreak\hfill\hskip 2em \bull\par\medbreak}

\def\bN{\mathbb{N}}

\def\bZ{\mathbb{Z}}
\def\cA{\mathcal{A}}

\def\de{\delta}

\def\om{\omega}
\def\si{\sigma}

\def\Om{\Omega}

\title{On the topological full group\\ containing the Grigorchuk group}
\author{Yaroslav Vorobets}
\date{}
\begin{document}

\maketitle

\begin{abstract}
We consider the topological full group of a substitution subshift induced 
by a substitution $a\to aca$, $b\to d$, $c\to b$, $d\to c$.  This group is 
interesting since the Grigorchuk group naturally embeds into it.  We show 
that the topological full group is finitely generated and give a simple 
generating set for it.
\end{abstract}

\section{Introduction}\label{intro}

Let $X$ be a Cantor set and $T:X\to X$ be a minimal homeomorphism.  The 
\emph{topological full group} of $T$, denoted $[[T]]$, is a transformation
group consisting of all homeomorphisms $f:X\to X$ that can be given by
$f(x)=T^{\nu(x)}(x)$, $x\in X$ for some continuous function $\nu:X\to\bZ$.
Continuity of $\nu$ implies that this function is locally constant and
takes only finitely many values.  Then nonempty level sets of $\nu$ form a
partition of the Cantor set $X$ into clopen (i.e., both closed and open)
sets.  Thus every element of $[[T]]$ is ``piecewise'' a power of $T$.  The 
topological full group $[[X]]$ is countable (as there are only countably 
many clopen subsets of $X$).

The notion of the topological full group was introduced by Giordano, Putnam 
and Skau \cite{GPS} who showed that $[[T]]$ is an (almost) complete 
invariant of $T$ as a topological dynamical system.

\begin{theorem}[{\cite{GPS}}]\label{full0}
Given minimal homeomorphisms $T_1:X\to X$ and $T_2:X\to X$ of a Cantor set 
$X$, the topological full groups $[[T_1]]$ and $[[T_2]]$ are isomorphic if 
and only if $T_1$ is topologically conjugate to $T_2$ or $T_2^{-1}$.
\end{theorem}

In this paper we are concerned with group-theoretical properties of a 
topological full group $[[T]]$.

\begin{theorem}[{\cite{GPS}}]\label{full1}
There exists a unique homomorphism $I:[[T]]\to\bZ$ such that $I(T)=1$.
\end{theorem}

The homomorphism $I$ is called the \emph{index map}.  Clearly, every 
element of finite order is contained in the kernel of $I$.

\begin{theorem}[{\cite{Mat}}]\label{full2}
The kernel of the index map is generated by elements of finite order.
\end{theorem}

Theorem \ref{full2} was proved by Matui \cite{Mat} who initiated the 
systematic study of group-theoretic properties of topological full 
groups (he also introduced the notation $[[T]]$).

We can construct many elements of finite order in $[[T]]$ as follows.  Let 
$U\subset X$ be a clopen set.  Suppose that for some integers $M$ and $N$, 
$M<N$, the sets $T^M(U),T^{M+1}(U),\dots,T^N(U)$ are pairwise disjoint.  
Then one can define a transformation $\Psi_{U,M,N}:X\to X$ by
$$
\Psi_{U,M,N}(x)=\left\{\!
\begin{array}{cl}
T(x)& \mbox{if }\, x\in T^M(U)\cup T^{M+1}(U)\cup\ldots\cup T^{N-1}(U),\\
T^{M-N}(x)& \mbox{if }\, x\in T^N(U),\\
x& \mbox{otherwise}.
\end{array}
\right.
$$
By construction, $\Psi_{U,M,N}$ is an element of the topological full group
$[[T]]$ of order $N-M+1$.  We are also going to use alternative notation
$\de_U$ for the map $\Psi_{U,0,1}$ and $\tau_U$ for the map $\Psi_{U,0,2}$. 
Each $\de_U$ is an element of order $2$ while each $\tau_U$ is an element 
of order $3$ (hence the notation: $\de$ as in $\de\acute\upsilon o$, 
$\tau$ as in $\tau\rho\acute\iota\alpha$).

It is not hard to show that every element of finite order in $[[T]]$ can be 
decomposed as a product of elements of the form $\de_U$.  Together with 
Theorems \ref{full1} and \ref{full2}, this yields the following.

\begin{theorem}[{\cite{Mat}}]\label{full3}
The topological full group $[[T]]$ is generated by $T$ and all 
transformations of the form $\de_U$.
\end{theorem}

There is much more to say about the commutator group of $[[T]]$.

\begin{theorem}[{\cite{Mat}}]\label{full4}
The commutator group of $[[T]]$ is generated by all elements of the form
$\tau_U$.
\end{theorem}

\begin{theorem}[{\cite{Mat}}]\label{full5}
The commutator group of $[[T]]$ is simple.
\end{theorem}

\begin{theorem}[{\cite{Mat}}]\label{full6}
The commutator group of $[[T]]$ is finitely generated if and only if $T$ is 
topologically conjugate to a (minimal) subshift.
\end{theorem}

Now we introduce a specific transformation $T$, a substitution subshift.  
Let $\si$ denote the \emph{Lysenok substitution} over the alphabet 
$\cA=\{a,b,c,d\}$, namely, $\si(a)=aca$, $\si(b)=d$, $\si(c)=b$, and 
$\si(d)=c$.  This substitution was originally used by Lysenok \cite{Lys} to 
obtain a nice recursive presentation of the Grigorchuk group:
$$
\mathcal{G}=\langle a,b,c,d\mid 
1=a^2=b^2=c^2=d^2=bcd=\si^k((ad)^4)=\si^k((adacac)^4),\,\, k\ge0\rangle.
$$
The substitution $\si$ acts naturally on the set $\cA^*$ of finite words 
over the alphabet $\cA$ as well as on the set $\cA^{\bN}$ of infinite words 
over $\cA$.  There exists a unique infinite word $\xi\in\cA^{\bN}$ fixed by 
$\si$: $\xi=acabacad\dots$  Let $T:\Om\to\Om$ be the two-sided subshift 
generated by $\xi$.  The phase space $\Om$ of the subshift $T$ consists of 
bi-infinite sequences $\om=\dots\om_{-2}\om_{-1}\om_0.\om_1\om_2\dots$ such 
that every finite subword $\om_l\om_{l+1}\dots\om_{m-1}\om_m$ occurs
somewhere in $\xi$.  The transformation is defined by
$T(\om)=\dots\om_{-1}\om_0\om_1.\om_2\om_3\dots$

\begin{theorem}[{\cite{Vor}}]\label{sub1}
The subshift $T$ is a minimal homeomorphism of the Cantor set $\Om$.
\end{theorem}

Given two finite words $u$ and $w$ over the alphabet $\cA$, we denote by
$[u.w]$ the set of all bi-infinite sequences
$\om=\dots\om_{-2}\om_{-1}\om_0.\om_1\om_2\dots$ in $\Om$ such that 
$\om_{-M+1}\dots\om_{-1}\om_0=u$ and $\om_1\om_2\dots\om_N=w$, where $M$ is
the length of $u$ and $N$ is the length of $w$ ($M,N\ge0$).  We refer to
$[u.w]$ as a \emph{cylinder of dimension} $M+N$.  The cylinder is a clopen
set.  Any clopen subset of $\Om$ splits into a disjoint union of cylinders
of dimension $N$ provided that $N$ is large enough.

The cylinder $[u.w]$ is a nonempty set if and only if the concatenated word 
$uw$ occurs in $\xi$ infinitely many times.  Infinitely many occurrences 
are required since $\xi$ is an infinite sequence while elements of $\Om$ 
are bi-infinite sequences.  Actually, $\xi$ is a Toeplitz sequence (see 
\cite{Vor} or Lemma \ref{seq2} below), which implies that every word 
occurring in $\xi$ does this infinitely often.
If at least one of the words $u$ and $w$ is not empty, then the cylinder
$[u.w]$ is disjoint from its image $T([u.w])$ (because there are no double
letters in $\xi$) so that the transformation $\de_{[u.w]}$ is well defined.

A direct relation between the Grigorchuk group $\mathcal{G}$ and the 
topological full group of the substitution subshift $T$ was established by 
Matte Bon \cite{M-B} who showed that $[[T]]$ contains a copy of 
$\mathcal{G}$.

\begin{theorem}[{\cite{M-B}}]\label{full-sub1}
The subgroup of $[[T]]$ generated by $\de_{[.a]}$, $\de_{[.b]}\de_{[.c]}$,
$\de_{[.c]}\de_{[.d]}$, and $\de_{[.d]}\de_{[.b]}$ is isomorphic to the
Grigorchuk group.
\end{theorem}

Theorem \ref{full6} implies that the commutator group of $[[T]]$ is 
finitely generated.  The main result of this paper is that the entire group 
$[[T]]$ is finitely generated.  Moreover, we provide an explicit generating 
set.

\begin{theorem}\label{main1}
The topological full group of the substitution subshift $T$ is generated by
transformations $T$, $\de_{[.b]}$, $\de_{[.d]}$, and $\de_{[.acacac]}$.
\end{theorem}

Note that $\de_{[.a]}$ and $\de_{[.c]}$ are not on the list of generators.
It turns out that
\begin{eqnarray*}
\de_{[.a]}&=&T^{-1}\de_{[.b]}\de_{[.d]}T^{-2}\de_{[.b]}\de_{[.d]}T^3
\de_{[.acacac]}T^2\de_{[.acacac]}T^{-2},\\
\de_{[.c]}&=&T^{-2}\de_{[.b]}\de_{[.d]}T^3\de_{[.acacac]}T^2
\de_{[.acacac]}T^{-3}
\end{eqnarray*}
(see Section \ref{tfg} for details).

The paper is organized as follows.  In Section \ref{seq} we obtain very 
detailed information on clopen subsets of the Cantor set $\Om$.  In Section 
\ref{tfg-gen} we derive some general properties of topological full groups 
(slightly generalizing \cite{Mat}).  In Section \ref{tfg} the results of 
Sections \ref{seq} and \ref{tfg-gen} are applied to prove Theorem 
\ref{main1}.  The proof is loosely modeled upon the proof of Theorem 
\ref{full6} in \cite{Mat}.

\section{Combinatorics of the substitution subshift}\label{seq}

First we are going to establish some properties of the infinite word
$\xi=\xi_1\xi_2\xi_3\dots$ fixed by the Lysenok substitution $\si$.

For any integer $n\ge1$ let $w_n=\si^{n-1}(a)$.  For example, $w_1=a$,
$w_2=aca$, $w_3=acabaca$, $w_4=acabacadacabaca$,
$w_5=acabacadacabacacacabacadacabaca$.  Since the word $w_1=a$ is
a proper beginning of the word $w_2=aca$, it follows by induction that each
$w_n$ is a proper beginning of $w_{n+1}$.  Consequently, there exists a
unique infinite word $\xi\in\cA^{\bN}$ such that each $w_n$ is a beginning
of $\xi$.  It is easy to see that $\xi$ is the only infinite word fixed by
$\si$.

For any integer $n\ge1$ let $l_n=\si^{n-1}(c)$.  Then $l_n=c$ if $n$ leaves
remainder $1$ under division by $3$, $l_n=b$ if $n$ leaves remainder $2$
under division by $3$, and $l_n=d$ if $n$ is divisible by $3$.

\begin{lemma}\label{seq1}
The word $w_n$ has length $2^n-1$ and $w_{n+1}=w_nl_nw_n$ for all $n\ge1$.
\end{lemma}

\begin{proof}
For any $n\ge1$ we obtain that $w_{n+1}=\si^n(a)=\si^{n-1}(\si(a))
=\si^{n-1}(aca)=\si^{n-1}(a)\si^{n-1}(c)\si^{n-1}(a)=w_nl_nw_n$.  Since the
word $w_1=a$ has length $1=2^1-1$, $l_n$ is always a single letter, and
$(2^n-1)+1+(2^n-1)=2^{n+1}-1$, it follows by induction that the length of
$w_n$ is $2^n-1$ for all $n\ge1$.
\end{proof}

\begin{lemma}\label{seq2}
Given an integer $N\ge1$, let $N=2^nK$, where $n\ge0$ and $K$ is odd.  Then
$\xi_N=a$ if $n=0$ and $\xi_N=l_n$ if $n\ge1$.
\end{lemma}

\begin{proof}
Let $S$ be the set of all words of the form $ar_1ar_2\dots ar_M$, where
each $r_i\in\{b,c,d\}$.  Since $\si(ab)=acad$, $\si(ac)=acab$, and
$\si(ad)=acac$, it follows that the set $S$ is invariant under the action
of the substitution $\si$.  Clearly, $ac\in S$.  Then $w_ml_m=\si^{m-1}(a)
\si^{m-1}(c)=\si^{m-1}(ac)$ is in $S$ for all $m\ge1$.  Since any beginning
of the infinite word $\xi$ is also a beginning of some $w_m$, we obtain
that $\xi=as_1as_2\dots$, where each $s_i\in\{b,c,d\}$.  In particular,
$\xi_N=a$ if and only if $N$ is odd.

Since the infinite word $\xi$ is fixed by the substitution $\si$, it 
follows that for any given $n\ge1$,
$$
\xi=\si^{n+1}(\xi)=\si^{n+1}(a)\si^{n+1}(s_1)\si^{n+1}(a)\si^{n+1}(s_2)
\ldots =w_{n+1}s'_1w_{n+1}s'_2\dots,
$$
where $s'_i=\si^{n+1}(s_i)$, $i=1,2,\dots$  Note that each $s'_i$ is a 
single letter from $\{b,c,d\}$.  By Lemma \ref{seq1}, $w_{n+1}=w_nl_nw_n$ 
and the length of $w_n$ is $2^n-1$.  Since $\xi=w_nl_nw_ns'_1w_nl_nw_ns'_2 
\dots$, we obtain that $\xi_N=l_n$ for $N=2^n,3\cdot2^n,5\cdot2^n,\dots$
That is, $\xi_N=l_n$ if $N=2^nK$, where $K$ is odd.
\end{proof}

\begin{lemma}\label{seq3}
$\si(\xi_{2N+1}\xi_{2N+2}\dots\xi_{2N+2M})=\xi_{4N+1}\xi_{4N+2}\dots
\xi_{4N+4M}$ for all $N\ge0$ and $M\ge1$.  Moreover, if
$\si(w)=\xi_{4N+1}\xi_{4N+2}\dots\xi_{4N+4M}$ for some $w$, then
$w=\xi_{2N+1}\xi_{2N+2}\dots\xi_{2N+2M}$.
\end{lemma}

\begin{proof}
For any $M\ge1$ the word $\xi_1\xi_2\dots\xi_{2M}$ is a beginning of the
infinite word $\xi$.  Since $\xi$ is invariant under the substitution
$\si$, the word $\si(\xi_1\xi_2\dots\xi_{2M})$ is another beginning of
$\xi$.  According to Lemma \ref{seq2}, $\xi_i=a$ if and only if $i$ is odd.
Hence the word $\xi_1\xi_2\dots\xi_{2M}$ contains $M$ letters $a$ and $M$
other letters.  Since $\si(a)=aca$ is a word of length $3$ while $\si(b)$,
$\si(c)$, and $\si(d)$ are single letters, the length of $\si(\xi_1\xi_2
\dots\xi_{2M})$ is $3M+M=4M$.  We conclude that $\si(\xi_1\xi_2\dots
\xi_{2M})=\xi_1\xi_2\dots\xi_{4M}$.  This proves the first statement of the
lemma in the case $N=0$.  In the case $N\ge1$, it follows from the above
that $\si(\xi_1\xi_2\dots\xi_{2N})=\xi_1\xi_2\dots\xi_{4N}$ and
$\si(\xi_1\xi_2\dots\xi_{2N+2M})=\xi_1\xi_2\dots\xi_{4N+4M}$.  Since
$$
\si(\xi_1\xi_2\dots\xi_{2N+2M})=\si(\xi_1\xi_2\dots\xi_{2N})\,
\si(\xi_{2N+1}\xi_{2N+2}\dots\xi_{2N+2M}),
$$
we obtain that $\si(\xi_{2N+1}\xi_{2N+2}\dots\xi_{2N+2M})=
\xi_{4N+1}\xi_{4N+2}\dots\xi_{4N+4M}$.

To prove the second statement of the lemma, it is enough to show that the
action of the substitution $\si$ on finite words is one-to-one, i.e.,
$\si(u_1)\ne\si(u_2)$ if $u_1\ne u_2$.  The reason is that neither of the
words $\si(a),\si(b),\si(c),\si(d)$ is a beginning of another (in
particular, neither is empty).  Let $u$ be the longest common beginning of
the words $u_1$ and $u_2$.  If $u=u_1$ then $u_2=u_1su'_2$ for some letter
$s$ and word $u'_2$.  Since $\si(s)$ is not empty, we obtain $\si(u_2)=
\si(u_1)\si(s)\si(u'_2)\ne\si(u_1)$.  The case $u=u_2$ is treated
similarly.  Otherwise $u_1=us_1u'_1$ and $u_2=us_2u'_2$, where $s_1,s_2$
are distinct letters and $u'_1,u'_2$ are some words.  It is no loss to
assume that the word $\si(s_1)$ is not longer than $\si(s_2)$.  Since
$\si(s_1)$ is not a beginning of $\si(s_2)$, it follows that
$\si(u)\si(s_1)$ is not a beginning of $\si(u)\si(s_2)\si(u'_2)=\si(u_2)$.
Then $\si(u_1)=\si(u)\si(s_1)\si(u'_1)$ cannot be the same as $\si(u_2)$.
\end{proof}

We say that a word $w'$ is obtained from a word $w$ by a \emph{cyclic
permutation} of letters if there exist words $u_1$ and $u_2$ such that
$w=u_1u_2$ and $w'=u_2u_1$.

\begin{lemma}\label{seq4}
Any word of length $2^n$ that occurs as a subword in $\xi$ can be obtained 
from one of the words $w_nb$, $w_nc$, and $w_nd$ by a cyclic permutation of 
letters.
\end{lemma}

\begin{proof}
Since the infinite word $\xi=\xi_1\xi_2\dots$ is invariant under the 
substitution $\si$, it follows that $\xi=\si^n(\xi)=\si^n(\xi_1) 
\si^n(\xi_2)\dots$  Lemma \ref{seq2} implies that $\xi_N=a$ if and only if 
$N$ is odd.  Therefore $\xi=w_ns_1w_ns_2w_ns_3\dots$, where each 
$s_i\in\{b,c,d\}$.  By Lemma \ref{seq1}, the length of the word $w_n$ is 
$2^n-1$.  It follows that any subword of length $2^n$ in $\xi$ is of the 
form $w_-lw_+$, where $l\in\{b,c,d\}$, $w_+$ is a beginning of $w_n$, and 
$w_-$ is an ending of $w_n$.  Since the concatenated word $w_+w_-$ has the 
same length as $w_n$, it has to coincide with $w_n$.  Then the word 
$w_-lw_+$ can be obtained from $w_nl$ by a cyclic permutation of letters.
\end{proof}

It turns out that the representation of the infinite word $\xi$ as
$w_ns_1w_ns_2w_ns_3\dots$, where each $s_i\in\{b,c,d\}$, does not show all
occurrences of $w_n$ as a subword in $\xi$.  There are more occurrences,
they overlap with the shown ones.  As a result, it is not true that any
occurrence of $w_n$ in $\xi$ is immediately followed by $bw_n$, $cw_n$, or
$dw_n$.  For example, some occurrences of $w_2=aca$ are followed by $caba$.
The next three lemmas explain what can follow and what can precede a
particular occurrence of $w_n$.

\begin{lemma}\label{seq-st1}
Any occurrence of the word $w_nl_n$ in $\xi$ is immediately followed by
$w_nb$, $w_nc$, or $w_nd$ and, unless it is the beginning of $\xi$,
immediately preceded by $w_nb$, $w_nc$, or $w_nd$.
\end{lemma}

\begin{proof}
The proof is by induction on $n$.  First consider the case $n=1$.  By Lemma
\ref{seq2}, $\xi_N=a$ if and only if $N$ is odd.  Therefore every occurrence
of $w_1l_1=ac$ in $\xi$ is immediately followed by $ab$, $ac$, or $ad$ and,
unless it is the beginning of $\xi$, immediately preceded by $ab$, $ac$, or
$ad$.

Now let $k\ge1$ and assume the lemma holds for $n=k$.  Suppose
$\xi_{N+1}\xi_{N+2}\dots\xi_{N+M}$ is an occurrence of the word
$w_{k+1}l_{k+1}$ in $\xi$.  The first four letters of $w_{k+1}l_{k+1}$ are
$acab$ so that $\xi_{N+4}=b$.  Lemma \ref{seq2} implies that $N+4$ is
divisible by $4$.  Besides, $M=2^{k+1}$ due to Lemma \ref{seq1}.  Hence $N$
and $M$ are both divisible by $4$, i.e., $N=4N'$ and $M=4M'$ for some
$M',N'\in\bZ$.  Since $\xi_{N+1}\xi_{N+2}\dots\xi_{N+M}=w_{k+1}l_{k+1}
=\si(w_kl_k)$, it follows from Lemma \ref{seq3} that
$\xi_{2N'+1}\xi_{2N'+2}\dots\xi_{2N'+2M'}$ is an occurrence of $w_kl_k$.  By
the inductive assumption, $\xi_{2N'+2M'+1}\xi_{2N'+2M'+2}\dots
\xi_{2N'+4M'}$ is an occurrence of $w_kb$, $w_kc$, or $w_kd$ and, unless
$N'=0$, we have $N'\ge M'$ and $\xi_{2N'-2M'+1}\xi_{2N'-2M'+2}\dots
\xi_{2N'}$ is also an occurrence of $w_kb$, $w_kc$, or $w_kd$.  Applying
Lemma \ref{seq3} two more times, we obtain that $\xi_{N+1}\xi_{N+2}\dots
\xi_{N+M}$ is immediately followed by $\si(w_kb)=w_{k+1}d$,
$\si(w_kc)=w_{k+1}b$, or $\si(w_kd)=w_{k+1}c$ and, unless $N=0$,
immediately preceded by one of the same three words.  This completes the
induction step.
\end{proof}

\begin{lemma}\label{seq-st2}
Any occurrence of the word $w_nl_{n+1}$ in $\xi$ is immediately followed by
$w_nl_n$ and immediately preceded by another $w_nl_n$.
\end{lemma}

\begin{proof}
The proof is by induction on $n$.  First consider the case $n=1$.  By Lemma
\ref{seq2}, $\xi_N=a$ if $N$ is odd and $\xi_N=c$ if $N$ is even while not
divisible by $4$.  It follows that every occurrence of $b$ or $d$ in $\xi$
is immediately followed and immediately preceded by $aca$.  Therefore every
occurrence of $w_1l_2=ab$ is immediately followed and preceded by
$ac=w_1l_1$.

Now let $k\ge1$ and assume the lemma holds for $n=k$.  Suppose
$\xi_{N+1}\xi_{N+2}\dots\xi_{N+M}$ is an occurrence of the word
$w_{k+1}l_{k+2}$ in $\xi$.  The first four letters of $w_{k+1}l_{k+2}$ are
$acab$ (if $k>1$) or $acad$ (if $k=1$).  In either case, Lemma \ref{seq2}
implies that $N+4$ is divisible by $4$.  Besides, $M=2^{k+1}$ due to Lemma
\ref{seq1}.  Hence $N$ and $M$ are both divisible by $4$, i.e., $N=4N'$ and
$M=4M'$ for some $M',N'\in\bZ$.  Since $\xi_{N+1}\xi_{N+2}\dots\xi_{N+M}
=w_{k+1}l_{k+2}=\si(w_kl_{k+1})$, it follows from Lemma \ref{seq3} that
$\xi_{2N'+1}\xi_{2N'+2}\dots\xi_{2N'+2M'}$ is an occurrence of $w_kl_{k+1}$.
By the inductive assumption, $N'\ge M'$ and
$$
\xi_{2N'+2M'+1}\xi_{2N'+2M'+2}\dots\xi_{2N'+4M'}
=\xi_{2N'-2M'+1}\xi_{2N'-2M'+2}\dots\xi_{2N'}=w_kl_k.
$$
Applying Lemma \ref{seq3} two more times, we obtain that
$\xi_{N+1}\xi_{N+2}\dots\xi_{N+M}$ is immediately followed and immediately
preceded by $\si(w_kl_k)=w_{k+1}l_{k+1}$.  This completes the induction
step.
\end{proof}

\begin{lemma}\label{seq-st3}
Any occurrence of the word $w_{n+1}l_n=w_nl_nw_nl_n$ in $\xi$ is 
immediately followed and immediately preceded by the same word of length 
$2^{n+1}$, which can be either $w_nl_nw_nl_{n+1}$ or $w_nl_{n+1}w_nl_n$.
\end{lemma}

\begin{proof}
The proof is by induction on $n$.  First consider the case $n=1$.  By Lemma
\ref{seq2}, $\xi_N=a$ if $N$ is odd, $\xi_N=c$ if $N$ is even but not
divisible by $4$, and $\xi_N=b$ if $N$ is divisible by $4$ but not by $8$.
Suppose $\xi_{M+1}\xi_{M+2}\xi_{M+3}\xi_{M+4}$ is an occurrence of
$w_2l_1=acac$ in $\xi$.  Then $M$ is even.  Moreover, the one of the
numbers $M+2$ and $M+4$ that is divisible by $4$ must be divisible by $8$
as well.  In particular, $M\ge4$.  If $M+4$ is divisible by $8$ then
$\xi_{M+5}\xi_{M+6}\xi_{M+7}\xi_{M+8}=\xi_{M-3}\xi_{M-2}\xi_{M-1}\xi_M=acab
=w_1l_1w_1l_2$.  If $M+2$ is divisible by $8$ then
$\xi_{M+5}\xi_{M+6}\xi_{M+7}\xi_{M+8}=\xi_{M-3}\xi_{M-2}\xi_{M-1}\xi_M
=abac=w_1l_2w_1l_1$.

Now let $k\ge1$ and assume the lemma holds for $n=k$.  Suppose
$\xi_{N+1}\xi_{N+2}\dots\xi_{N+M}$ is an occurrence of the word
$w_{k+2}l_{k+1}$ in $\xi$.  The first four letters of $w_{k+2}l_{k+1}$ are
$acab$ so that $\xi_{N+4}=b$.  Lemma \ref{seq2} implies that $N+4$ is
divisible by $4$.  Besides, $M=2^{k+2}$ due to Lemma \ref{seq1}.  Hence $N$
and $M$ are both divisible by $4$, i.e., $N=4N'$ and $M=4M'$ for some
$M',N'\in\bZ$.  Since $\xi_{N+1}\xi_{N+2}\dots\xi_{N+M}=w_{k+2}l_{k+1}
=\si(w_{k+1}l_k)$, it follows from Lemma \ref{seq3} that
$\xi_{2N'+1}\xi_{2N'+2}\dots\xi_{2N'+2M'}$ is an occurrence of $w_{k+1}l_k$.
By the inductive assumption, this occurrence is immediately followed and
immediately preceded by the same word $u$ of even length, where
$u=w_kl_kw_kl_{k+1}$ or $u=w_kl_{k+1}w_kl_k$.  Applying Lemma \ref{seq3}
two more times, we obtain that $\xi_{N+1}\xi_{N+2}\dots\xi_{N+M}$ is
immediately followed and immediately preceded by $\si(u)$.  Note that
$\si(u)=\si(w_kl_kw_kl_{k+1})=w_{k+1}l_{k+1}w_{k+1}l_{k+2}$ or
$\si(u)=\si(w_kl_{k+1}w_kl_k)=w_{k+1}l_{k+2}w_{k+1}l_{k+1}$.  The length of
$\si(u)$ is $2^{k+2}$ due to Lemma \ref{seq1}.  This completes the
induction step.
\end{proof}

Next we are going to derive some properties of the cylinders in $\Om$.

\begin{lemma}\label{cyl1}
Let $n\ge2$.  Then the cylinder $[.w_n]$ is disjoint from $T^N([.w_n])$ for
$1\le N<2^{n-1}$.
\end{lemma}

\begin{proof}
Let $n\ge2$ and suppose the cylinder $[.w_n]$ is not disjoint from
$T^N([.w_n])$ for some $N\ge1$.  We need to show that $N\ge 2^{n-1}$.  Take
any element $\om=\dots\om_{-2}\om_{-1}\om_0.\om_1\om_2\dots$ of the
intersection  $[.w_n]\cap T^N([.w_n])$.  Then $\om$ and $T^{-N}(\om)$ are
both in $[.w_n]$.  By construction of the infinite word
$\xi=\xi_1\xi_2\dots$, the word $w_n$ is a beginning of $\xi$.  The length
of $w_n$ is $2^n-1$ due to Lemma \ref{seq1}.  Since $\om$ and $T^{-N}(\om)$
belong to $[.w_n]$, it follows that $\om_i=\om_{i-N}=\xi_i$ for $1\le i\le
2^n-1$.  As a consequence, $\xi_i=\xi_{i+N}$ whenever $1\le i<i+N\le2^n-1$.

The integer $N$ is uniquely represented as $N=2^kK$, where $k\ge0$ and $K$
is odd.  By Lemma \ref{seq2}, $\xi_N=l_k$ if $k\ge1$ and $\xi_N=a$ if
$k=0$.  By the same lemma, $\xi_{2N}=l_{k+1}$, which implies that $\xi_N\ne
\xi_{2N}=\xi_{N+N}$.  Then it follows from the above that $2N>2^n-1$.
Since $N$ is an integer, this is equivalent to $2N\ge 2^n$ or $N\ge
2^{n-1}$.
\end{proof}

\begin{lemma}\label{cyl2-pre}
$[l.w_nl_n]=[w_nl.w_nl_nw_n]$ and $[l_n.w_nl]=[w_nl_n.w_nlw_n]$ for all 
$l\in\{b,c,d\}$ and $n\ge1$.  Moreover, if $l\ne l_n$ then 
$[l.w_nl_n]=[w_nl_nw_nl.]$.
\end{lemma}

\begin{proof}
By Lemma \ref{seq-st1}, any occurrence of the word $w_nl_n$ in $\xi$ is
immediately followed by $w_nb$, $w_nc$, or $w_nd$ and, unless it is the
beginning of $\xi$, immediately preceded by $w_nb$, $w_nc$, or $w_nd$.  As
a consequence, any occurrence of $lw_nl_n$ is immediately followed and
preceded by $w_n$.  This implies an equality of cylinders $[l.w_nl_n]=
[w_nl.w_nl_n]=[w_nl.w_nl_nw_n]$.  In the case $l=l_n$, we are done.  In the
other cases, two more equalities are to be derived.

Next consider the case $l=l_{n+1}$.  By Lemma \ref{seq-st2}, any occurrence
of the word $w_nl_{n+1}$ in $\xi$ is immediately followed and preceded by
$w_nl_n$.  Therefore any occurrence of $l_nw_nl$ is immediately followed
and preceded by $w_n$ so that $[l_n.w_nl]=[w_nl_n.w_nlw_n]$.  Besides, this
implies that $[w_nl_nw_nl.]=[w_nl.]=[w_nl.w_nl_n]$.  We already know that
$[w_nl.w_nl_n]=[l.w_nl_n]$.

Now consider the case when $l\ne l_n$, $l\ne l_{n+1}$, and $n\ge2$.  In
this case, $l=l_{n-1}$.  By Lemma \ref{seq-st3}, any occurrence of the word
$w_nl_{n-1}=w_{n-1}l_{n-1}w_{n-1}l_{n-1}$ in $\xi$ is immediately followed
and preceded by the same word of length $2^n$, which can be either
$w_{n-1}l_{n-1}w_{n-1}l_n=w_nl_n$ or $w_{n-1}l_nw_{n-1}l_{n-1}$.  Therefore
any occurrence of $l_nw_nl$ is immediately followed and preceded by $w_n$
so that $[l_n.w_nl]=[w_nl_n.w_nlw_n]$.  Another consequence is that
$[w_nl_nw_nl.]=[w_nl_nw_nl.w_nl_n]=[w_nl.w_nl_n]$ (unlike the previous
case, this does not equal $[w_nl.]$).  We already know that
$[w_nl.w_nl_n]=[l.w_nl_n]$.

It remains to consider the case when $l\ne l_n$, $l\ne l_{n+1}$, and $n=1$.
In this case, $w_n=a$, $l_n=c$, and $l=d$.  As already observed in the
proof of Lemma \ref{seq-st2}, every occurrence of the letter $d$ in $\xi$
is immediately followed and preceded by $aca$.  Therefore $[c.ad]=[ac.ada]$
and $[d.ac]=[acad.aca]=[acad.]$.
\end{proof}

The following lemma is crucial for the proof of Theorem \ref{main1}.

\begin{lemma}\label{cyl2}
If $C$ is a nonempty cylinder of dimension $2^n$, then $C=T^N([.w_nl])$ for
some $l\in\{b,c,d\}$ and $N\in\bZ$.
\end{lemma}

\begin{proof}
Let $C$ be a nonempty cylinder of dimension $2^n$.  We have $C=[w_-.w_+]$ 
for some words $w_-$ and $w_+$ such that the concatenated word $w=w_-w_+$ 
has length $2^n$.  Since $C$ is a nonempty set, the word $w$ must occur as 
a subword in the infinite word $\xi$.  By Lemma \ref{seq4}, $w$ can be 
obtained from a word $w_nl$, $l\in\{b,c,d\}$, by a cyclic permutation of 
letters.  We are going to show that $[.w]=T^N([.w_nl])$ for some 
$N\in\bZ$.  Then $C=T^M([.w])=T^{M+N}([.w_nl])$, where $M$ is the length of 
$w_-$.

First consider the case $n=1$.  In this case, $w=w_1l=al$ or $w=la$.  By
Lemma \ref{seq2},  $\xi_N=a$ if and only if $N$ is odd.  Hence every
occurrence of the letter $l$ in $\xi$ is immediately followed and preceded
by $a$.  Therefore $[.la]=[.l]=[a.l]=T([.al])$.

Now assume $n\ge2$.  In this case, $w_n=w_{n-1}l_{n-1}w_{n-1}$ due to Lemma
\ref{seq1}.  Let $u_1$ and $u_2$ be words such that $w_nl=u_1u_2$ and
$w=u_2u_1$.  If $u_2$ is longer than $u_1$, then $l_{n-1}w_{n-1}l$ is an
ending of $u_2$, i.e., $u_2=u'_2l_{n-1}w_{n-1}l$ for some word $u'_2$.
Clearly, $w=u'_2l_{n-1}w_{n-1}lu_1$ and $u_1u'_2=w_{n-1}$.  If $u_2$ is not
longer than $u_1$ and not empty, we have $u_1=w_{n-1}l_{n-1}u'_1$ and
$u_2=u'_2l$ for some words $u'_1$ and $u'_2$.  Then
$w=u'_2lw_{n-1}l_{n-1}u'_1$ and $u'_1u'_2=w_{n-1}$.  Finally, if $u_2$ is
empty, then $w=w_nl=w_{n-1}l_{n-1}w_{n-1}l=u'_2l_{n-1}w_{n-1}lu'_1$, where
$u'_1$ is the empty word and $u'_2=w_{n-1}$.

By the above the word $w$ can be represented as $u'_2l_{n-1}w_{n-1}lu'_1$
or $u'_2lw_{n-1}l_{n-1}u'_1$, where the words $u'_1$ and $u'_2$ satisfy
$u'_1u'_2=w_{n-1}$.  Note that both representations are the same if
$l=l_{n-1}$ (also, in this case there are two different choices for the
pair $u_1,u_2$).  First assume that $w=u'_2l_{n-1}w_{n-1}lu'_1$.  Since
$u'_1$ is a beginning of $w_{n-1}$ and $u'_2$ is an ending of $w_{n-1}$, it
follows that $[w_{n-1}l_{n-1}.w_{n-1}lw_{n-1}]\subset
[u'_2l_{n-1}.w_{n-1}lu'_1]\subset[l_{n-1}.w_{n-1}l]$.  Similarly,
$[w_{n-1}l_{n-1}.w_{n-1}lw_{n-1}]\subset[w_{n-1}l_{n-1}.w_{n-1}l]\subset
[l_{n-1}.w_{n-1}l]$.  Since $[w_{n-1}l_{n-1}.w_{n-1}lw_{n-1}]=
[l_{n-1}.w_{n-1}l]$ due to Lemma \ref{cyl2-pre}, we obtain that
$[u'_2l_{n-1}.w_{n-1}lu'_1]=[w_{n-1}l_{n-1}.w_{n-1}l]$.  The latter
equality can be rewritten as $T^{N_1}([.w])=T^{N_2}([.w_nl])$, where $N_1$
is the length of $u'_2l_{n-1}$ and $N_2$ is the length of $w_{n-1}l_{n-1}$
($N_2=2^{n-1}$).  Then $[.w]=T^{N_2-N_1}([.w_nl])$.

Now assume that $l\ne l_{n-1}$ and $w=u'_2lw_{n-1}l_{n-1}u'_1$.  Just like
in the previous case, we obtain that $[w_{n-1}l.w_{n-1}l_{n-1}w_{n-1}]
\subset[u'_2l.w_{n-1}l_{n-1}u'_1]\subset[l.w_{n-1}l_{n-1}]$.  By Lemma
\ref{cyl2-pre}, $[w_{n-1}l.w_{n-1}l_{n-1}w_{n-1}]=[l.w_{n-1}l_{n-1}]
=[w_{n-1}l_{n-1}w_{n-1}l.]$.  It follows that $[u'_2l.w_{n-1}l_{n-1}u'_1]
=[w_{n-1}l_{n-1}w_{n-1}l.]$.  The latter equality can be rewritten as
$T^{N_1}([.w])=T^{N_2}([.w_nl])$, where $N_1$ is the length of $u'_2l$ and
$N_2$ is the length of $w_nl$ ($N_2=2^n$).  Then
$[.w]=T^{N_2-N_1}([.w_nl])$.
\end{proof}

The next three lemmas establish relations between cylinders of dimension 
$2^n$ and cylinders of dimension $2^{n+1}$.  We shall use $\sqcup$ to 
denote disjoint unions.  Namely, $U=U_1\sqcup U_2\sqcup\ldots\sqcup U_k$ 
means that $U=U_1\cup U_2\cup\ldots\cup U_k$ and the sets 
$U_1,U_2,\dots,U_k$ are pairwise disjoint. 

\begin{lemma}\label{cyl-st1}
$[.w_nl_n]=[.w_{n+1}b]\sqcup[.w_{n+1}c]\sqcup[.w_{n+1}d]$ for all $n\ge1$.
\end{lemma}

\begin{proof}
The cylinders $[.w_{n+1}b]$, $[.w_{n+1}c]$, and $[.w_{n+1}d]$ are clearly
disjoint.  Since $w_{n+1}=w_nl_nw_n$ (due to Lemma \ref{seq1}), each of
them is contained in $[.w_nl_n]$.  Lemma \ref{seq-st1} implies that the
union of the three cylinders is exactly $[.w_nl_n]$. 
\end{proof}

\begin{lemma}\label{cyl-st2}
$[.w_nl_{n+1}]=T^{2^n}([.w_{n+1}l_{n+1}])$ for all $n\ge1$.
\end{lemma}

\begin{proof}
Lemma \ref{seq-st2} implies that $[.w_nl_{n+1}]=[w_nl_n.w_nl_{n+1}]$.
Since the length of the word $w_nl_n$ is $2^n$, we obtain that
$[.w_nl_{n+1}]=T^{2^n}([.w_nl_nw_nl_{n+1}])$.  It remains to notice that
$w_nl_nw_n=w_{n+1}$.
\end{proof}

\begin{lemma}\label{cyl-st3}
$[.w_{n+1}l_n]=T^{2^{n+1}}([.w_{n+2}l_n])\sqcup
T^{3\cdot2^n}([.w_{n+2}l_n])$ for all $n\ge1$.
\end{lemma}

\begin{proof}
Lemma \ref{seq-st3} implies that the cylinder
$[.w_{n+1}l_n]=[.w_nl_nw_nl_n]$ is the union of cylinders
$C_1=[w_nl_nw_nl_{n+1}.w_nl_nw_nl_n]$ and
$C_2=[w_nl_{n+1}w_nl_n.w_nl_nw_nl_n]$, which are disjoint since $l_{n+1}$ 
is always different from $l_n$.  Lemma \ref{seq-st2} further implies that
$C_2=[w_nl_nw_nl_{n+1}w_nl_n.w_nl_nw_nl_n]$.  Then it follows from Lemma
\ref{seq-st3} that $C_2=[w_nl_nw_nl_{n+1}w_nl_n.w_nl_n]$.

Notice that $w_nl_nw_nl_{n+1}w_nl_nw_nl_n=w_{n+1}l_{n+1}w_{n+1}l_n
=w_{n+2}l_n$.  Since the word $w_nl_nw_nl_{n+1}$ has length $2^{n+1}$ and
the word $w_nl_nw_nl_{n+1}w_nl_n$ has length $3\cdot 2^n$, we obtain that
$C_1=T^{2^{n+1}}([.w_{n+2}l_n])$ and $C_2=T^{3\cdot2^n}([.w_{n+2}l_n])$.
\end{proof}

\section{General topological full group}\label{tfg-gen}

We proceed to the study of the topological full group $[[T]]$.  Let us 
begin with some general properties of transformations $\Psi_{U,M,N}$ that 
hold for any homeomorphism $T:X\to X$ of a Cantor set $X$ onto itself.

\begin{lemma}\label{psi-gen1}
If $\Psi_{U,M,N}$ is well defined for a clopen set $U$ and integers $M,N$, 
$M<N$, then $\Psi_{T^K(U),M+J,N+J}$ is well defined for any $J,K\in\bZ$ and 
$\Psi_{T^K(U),M+J,N+J}=T^{J+K}\Psi_{U,M,N}T^{-J-K}$.
\end{lemma}

\begin{proof}
Since $\Psi_{U,M,N}$ is well defined, the sets $T^M(U),T^{M+1}(U),\dots,
T^N(U)$ are pairwise disjoint.  Since $T$ is an invertible transformation,
it follows that for any $J\in\bZ$ the sets $T^{M+J}(U),T^{M+J+1}(U),\dots,
T^{N+J}(U)$ are also pairwise disjoint.  Hence $\Psi_{U,M+J,N+J}$ is
defined as well.  Suppose $x\in X$ and let $y=T^{-J}(x)$.  Then
$\Psi_{U,M+J,N+J}(x)=T^n(x)$ for a specific $n$ (which can be $0$, $1$, or 
$M-N$) if and only if $\Psi_{U,M,N}(y)=T^n(y)$.  It follows that
$\Psi_{U,M+J,N+J}=T^J\Psi_{U,M,N}T^{-J}$.

Given $K\in\bZ$, let $V=T^K(U)$.  Then $V$ is a clopen set and
$T^i(V)=T^{i+K}(U)$ for all $i\in\bZ$.  It follows that $\Psi_{V,M',N'}=
\Psi_{U,M'+K,N'+K}$ whenever one of these transformations is defined.  In
particular, $\Psi_{V,M+J,N+J}=\Psi_{U,M+J+K,N+J+K}$ for all $J\in\bZ$.  By
the above, $\Psi_{U,M+J+K,N+J+K}=T^{J+K}\Psi_{U,M,N}T^{-J-K}$.
\end{proof}

\begin{lemma}\label{psi-gen2}
Suppose $\Psi_{U,M,N}$ is well defined and $U=U_1\sqcup U_2\sqcup\ldots 
\sqcup U_k$, where $U_1,U_2,\dots,U_k$ are clopen sets.  Then 
transformations $\Psi_{U_i,M,N}$, $1\le i\le k$ are also well defined, they 
commute with one another, and $\Psi_{U,M,N}=\Psi_{U_1,M,N}\Psi_{U_2,M,N} 
\dots\Psi_{U_k,M,N}$.
\end{lemma}

\begin{proof}
Since $\Psi_{U,M,N}$ is well defined, the sets $T^M(U),T^{M+1}(U),\dots,
T^N(U)$ are pairwise disjoint.  Since each $U_i$ is a subset of $U$, the
sets $T^M(U_i),T^{M+1}(U),\dots,T^N(U_i)$ are also pairwise disjoint.
Hence $\Psi_{U_i,M,N}$ is defined as well.  The transformation
$\Psi_{U_i,M,N}$ coincides with $\Psi_{U,M,N}$ on the set $\widetilde U_i
=T^M(U_i)\cup T^{M+1}(U_i)\cup\dots\cup T^N(U_i)$ and with the identity map
anywhere else.  Since $U=U_1\sqcup U_2\sqcup\ldots\sqcup U_k$, it follows
that $T^J(U)=T^J(U_1)\sqcup T^J(U_2)\sqcup\ldots\sqcup T^J(U_k)$ for all
$J\in\bZ$.  As a consequence, $T^M(U)\cup T^{M+1}(U)\cup\dots\cup T^N(U)
=\widetilde U_1\sqcup \widetilde U_2\sqcup\ldots\sqcup \widetilde U_k$.
This implies that transformations $\Psi_{U_1,M,N},\Psi_{U_2,M,N},\dots,
\Psi_{U_k,M,N}$ commute with one another and $\Psi_{U_1,M,N}\Psi_{U_2,M,N}
\dots\Psi_{U_k,M,N}=\Psi_{U,M,N}$.
\end{proof}

\begin{lemma}\label{psi-gen3}
If $\Psi_{U,M,N}$ is well defined and $N-M\ge2$, then
$\Psi_{U,M,N}=\Psi_{U,M,K}\Psi_{U,K,N}$ for any $K$, $M<K<N$.
\end{lemma}

\begin{proof}
Since $\Psi_{U,M,N}$ is well defined, the sets $T^M(U),T^{M+1}(U),\dots,
T^N(U)$ are pairwise disjoint.  It follows that transformations 
$\Psi_{U,M,K}$ and $\Psi_{U,K,N}$ are well defined for any $K$, $M<K<N$.  
We need to show that $\Psi_{U,M,N}(x)=\Psi_{U,M,K}(\Psi_{U,K,N}(x))$ for
all $x\in X$.  First consider the case $x\in T^i(U)$, where $M\le i\le K-1$.
Then $x\notin T^j(U)$ for $K\le j\le N$.  Hence $x$ is fixed by 
$\Psi_{U,K,N}$.  Consequently, $\Psi_{U,M,K}(\Psi_{U,K,N}(x))= 
\Psi_{U,M,K}(x)=T(x)$, which coincides with $\Psi_{U,M,N}(x)$.

Next consider the case $x\in T^i(U)$, where $K\le i\le N-1$.  In this case, 
$\Psi_{U,K,N}(x)=T(x)$.  Since $T(x)\in T^{i+1}(U)$ and $K+1\le i+1\le N$, 
it follows that $T(x)\notin T^j(U)$ for $M\le j\le K$.  Hence $T(x)$ is 
fixed by $\Psi_{U,M,K}$ so that $\Psi_{U,M,K}(\Psi_{U,K,N}(x))=T(x) 
=\Psi_{U,M,N}(x)$.

Now consider the case $x\in T^N(U)$.  In this case,
$\Psi_{U,K,N}(x)=T^{K-N}(x)$, which belongs to $T^K(U)$.  Then
$\Psi_{U,M,K}(\Psi_{U,K,N}(x))=\Psi_{U,M,K}(T^{K-N}(x))=T^{M-K}(T^{K-N}(x))
=T^{M-N}(x)$, which coincides with $\Psi_{U,M,N}(x)$.

Finally, if $x\notin T^i(U)$ for all $i$, $M\le i\le N$, then $x$ is fixed
by all three transformations.  In particular,
$\Psi_{U,M,K}(\Psi_{U,K,N}(x))=x=\Psi_{U,M,N}(x)$.
\end{proof}

\begin{lemma}\label{psi-gen4}
Suppose $\Psi_{U,M,K}$ and $\Psi_{V,K,N}$ are well defined.  If $T^i(V)\cap 
U=\emptyset$ for $1\le i\le N-M$, then $\Psi_{V,K,N}\Psi_{U,M,K}^{-1} 
\Psi_{V,K,N}^{-1}\Psi_{U,M,K}=\Psi_{U\cap V,K-1,K+1}$.
\end{lemma}

\begin{proof}
Since $\Psi_{U,M,K}$ is well defined, the sets $T^M(U),T^{M+1}(U),\dots,
T^K(U)$ are pairwise disjoint.  Since $\Psi_{V,K,N}$ is well defined, the
sets $T^K(V),T^{K+1}(V),\dots,T^N(V)$ are pairwise disjoint.  Further,
$T^i(U)\cap T^j(V)=T^i(U\cap T^{j-i}(V))$ for all $i,j\in\bZ$.  Therefore
$T^i(U)$ is disjoint from $T^j(V)$ whenever $1\le j-i\le N-M$.  In
particular, the two sets are disjoint if $M\le i\le K\le j\le N$ and at
least one of the numbers $i$ and $j$ is different from $K$.  It follows
that sets $T^M(U),T^{M+1}(U),\dots,T^{K-1}(U),T^K(U)\cup T^K(V)=T^K(U\cup
V)$, $T^{K+1}(V),\dots,T^{N-1}(V),T^N(V)$ are pairwise disjoint.

Let $W=U\cap V$, $Y=U\setminus W$, and $Z=V\setminus W$.  Then $W$, $Y$,
and $Z$ are clopen sets.  We have $U=W\sqcup Y$, $V=W\sqcup Z$, and $U\cup
V=W\sqcup Y\sqcup Z$.  By Lemma \ref{psi-gen2}, $\Psi_{U,M,K}=\Psi_{W,M,K}
\Psi_{Y,M,K}$ and $\Psi_{V,K,N}=\Psi_{W,K,N}\Psi_{Z,K,N}$.  The
transformation $\Psi_{Y,M,K}$ moves points only within the set $\widetilde
Y=T^M(Y)\cup T^{M+1}(Y)\cup\dots\cup T^K(Y)$.  Likewise, $\Psi_{Z,K,N}$
moves points only within the set $\widetilde Z=T^K(Z)\cup T^{K+1}(Z)\cup
\dots\cup T^N(Z)$.  The transformations $\Psi_{W,M,K}$ and $\Psi_{W,K,N}$
do not move points outside of the set $\widetilde W=T^M(W)\cup T^{M+1}(W)
\cup\dots\cup T^N(W)$.  Note that $T^i(U)=T^i(W)\sqcup T^i(Y)$ for $M\le
i\le K-1$, $T^i(V)=T^i(W)\sqcup T^i(Z)$ for $K+1\le i\le N$, and $T^K(U\cup
V)=T^K(W)\sqcup T^K(Y)\sqcup T^K(Z)$.  It follows that the sets $\widetilde
W$, $\widetilde Y$, and $\widetilde Z$ are pairwise disjoint.  This implies
that the transformations $\Psi_{Y,M,K}$ and $\Psi_{Z,K,N}$ commute with
$\Psi_{W,M,K}$, $\Psi_{W,K,N}$, and with each other.  Then
$$
\Psi_{V,K,N}\Psi_{U,M,K}^{-1}\Psi_{V,K,N}^{-1}\Psi_{U,M,K}=
$$ $$
=(\Psi_{W,K,N}\Psi_{Z,K,N}) (\Psi_{W,M,K}\Psi_{Y,M,K})^{-1}
(\Psi_{W,K,N}\Psi_{Z,K,N})^{-1} (\Psi_{W,M,K}\Psi_{Y,M,K})
$$ $$
=\Psi_{W,K,N}\Psi_{Z,K,N} \Psi_{Y,M,K}^{-1}\Psi_{W,M,K}^{-1}
\Psi_{Z,K,N}^{-1}\Psi_{W,K,N}^{-1} \Psi_{W,M,K}\Psi_{Y,M,K}
$$ $$
=\Psi_{W,K,N}\Psi_{W,M,K}^{-1} (\Psi_{Z,K,N}\Psi_{Y,M,K}^{-1}
\Psi_{Z,K,N}^{-1}\Psi_{Y,M,K}) \Psi_{W,K,N}^{-1} \Psi_{W,M,K}
$$ $$
=\Psi_{W,K,N}\Psi_{W,M,K}^{-1}\Psi_{W,K,N}^{-1} \Psi_{W,M,K}.
$$

Let $L=\Psi_{W,M,K-1}$ if $M<K-1$ and let $L$ be the identity map
otherwise.  Let $R=\Psi_{W,K+1,N}$ if $K+1<N$ and let $R$ be the identity
map otherwise.  It follows from Lemma \ref{psi-gen3} that $\Psi_{W,M,K}=
L\Psi_{W,K-1,K}$ and $\Psi_{W,K,N}=\Psi_{W,K,K+1}R$.  The transformation 
$L$ fixes all points in the set $T^K(W)\cup T^{K+1}(W)\cup\dots\cup 
T^N(W)$, which implies that $L$ commutes with $\Psi_{W,K,K+1}$ and $R$.  
Similarly, $R$ fixes all points in the set $T^M(W)\cup T^{M+1}(W)\cup\dots 
\cup T^K(W)$, which implies that $R$ commutes with $\Psi_{W,K-1,K}$ and 
$L$.  Then
$$
\Psi_{W,K,N}\Psi_{W,M,K}^{-1}\Psi_{W,K,N}^{-1}\Psi_{W,M,K}
=(\Psi_{W,K,K+1}R) (L\Psi_{W,K-1,K})^{-1} (\Psi_{W,K,K+1}R)^{-1} 
(L\Psi_{W,K-1,K})
$$ $$
=\Psi_{W,K,K+1}R \Psi_{W,K-1,K}^{-1}L^{-1} R^{-1}\Psi_{W,K,K+1}^{-1}
L\Psi_{W,K-1,K}
$$ $$
=\Psi_{W,K,K+1} \Psi_{W,K-1,K}^{-1} (RL^{-1}R^{-1}L) \Psi_{W,K,K+1}^{-1}
\Psi_{W,K-1,K}
$$ $$
=\Psi_{W,K,K+1}\Psi_{W,K-1,K}^{-1}\Psi_{W,K,K+1}^{-1}\Psi_{W,K-1,K}.
$$
Since $\Psi_{W,K-1,K}$ and $\Psi_{W,K,K+1}$ are involutions, we obtain that
$$
\Psi_{W,K,K+1}\Psi_{W,K-1,K}^{-1}\Psi_{W,K,K+1}^{-1}\Psi_{W,K-1,K}
=(\Psi_{W,K,K+1}\Psi_{W,K-1,K})^2=(\Psi_{W,K-1,K}\Psi_{W,K,K+1})^{-2}.
$$
It follows from the above that sets $T^M(W),\dots,T^{K-1}(W),T^K(W), 
T^{K+1}(W),\dots,T^N(W)$ are pairwise disjoint.  In particular, the 
transformation $\Psi_{W,K-1,K+1}$ is well defined.  We have 
$\Psi_{W,K-1,K}\Psi_{W,K,K+1}=\Psi_{W,K-1,K+1}$ due to Lemma 
\ref{psi-gen3} and $\Psi^{-2}_{W,K-1,K+1}=\Psi_{W,K-1,K+1}$ since 
$\Psi_{W,K-1,K+1}$ has order $3$.
\end{proof}

\section{Topological full group of the substitution subshift}\label{tfg}

Now we restrict our attention to the substitution subshift $T:\Om\to\Om$.
Let $G$ be the subgroup of $[[T]]$ generated by transformations $T$,
$\de_{[.b]}$, $\de_{[.d]}$, and $\de_{[.acacac]}$.  For any $n\ge1$ let
$G_n$ be the subgroup of $[[T]]$ generated by $\de_{[.w_nb]}$,
$\de_{[.w_nc]}$, $\de_{[.w_nd]}$, and $T$.

\begin{lemma}\label{psi-de-3}
$G_3=G$.
\end{lemma}

\begin{proof}
First we show that the group $G_3$ contains $\de_{[.w_2b]}$.  By Lemma
\ref{cyl-st1}, $[.w_2b]=[.w_3b]\sqcup[.w_3c]\sqcup[.w_3d]$.  Then Lemma
\ref{psi-gen2} implies that $\de_{[.w_2b]}=\de_{[.w_3b]}\de_{[.w_3c]}
\de_{[.w_3d]}$.

By Lemma \ref{seq2}, $\xi_i=a$ if $i$ is odd, $\xi_i=c$ if $i$ is even but
not divisible by $4$, and $\xi_i=b$ if $i$ is divisible by $4$ but not by
$8$.  It follows that every occurrence of the letter $b$ in $\xi$ is
immediately preceded by $aca$ while every occurrence of $d$ is preceded by
$acabaca$.  As a consequence, $[.b]=[aca.b]=T^3([.w_2b])$ and
$[.d]=[acabaca.d]=T^7([.w_3d])$.  Besides, Lemma \ref{seq-st3} implies that 
$[.acacac]=[acab.acacacab]=[acab.acac]=T^4([.w_3c])$.  By Lemma
\ref{psi-gen1}, $\de_{[.b]}=T^3\de_{[.w_2b]}T^{-3}$, $\de_{[.d]}
=T^7\de_{[.w_3d]}T^{-7}$, and $\de_{[.acacac]}=T^4\de_{[.w_3c]}T^{-4}$.
Therefore all generators of the group $G$ belong to $G_3$ so that $G\subset
G_3$.

Conversely, it follows from the above that $\de_{[.w_2b]}
=T^{-3}\de_{[.b]}T^3$, $\de_{[.w_3c]}=T^{-4}\de_{[.acacac]}T^4$,
$\de_{[.w_3d]}=T^{-7}\de_{[.d]}T^7$, and
\begin{eqnarray*}
\de_{[.w_3b]}&=&\de_{[.w_2b]}\de_{[.w_3d]}^{-1}\de_{[.w_3c]}^{-1}
=\de_{[.w_2b]}\de_{[.w_3d]}\de_{[.w_3c]}\\
&=&(T^{-3}\de_{[.b]}T^3)(T^{-7}\de_{[.d]}T^7)(T^{-4}\de_{[.acacac]}T^4)
=T^{-3}\de_{[.b]}T^{-4}\de_{[.d]}T^3\de_{[.acacac]}T^4.
\end{eqnarray*}
Therefore all generators of the group $G_3$ belong to $G$ so that
$G_3\subset G$.
\end{proof}

As a follow-up to the previous proof, let us derive the formulas for
$\de_{[.a]}$ and $\de_{[.c]}$.  We begin with some auxiliary formulas.  By
Lemma \ref{cyl-st3}, $[.acac]=T^4([.w_3c])\sqcup T^6([.w_3c])$.  Then
Lemmas \ref{psi-gen1} and \ref{psi-gen2} imply that
$$
\de_{[.acac]}=(T^4\de_{[.w_3c]}T^{-4})(T^6\de_{[.w_3c]}T^{-6})
=T^4\de_{[.w_3c]}T^2\de_{[.w_3c]}T^{-6}.
$$
Since $\de_{[.w_3c]}=T^{-4}\de_{[.acacac]}T^4$, we obtain that
$\de_{[.acac]}=\de_{[.acacac]}T^2\de_{[.acacac]}T^{-2}$.  Further,
$[.acad]=T^4([.w_3d])$ due to Lemma \ref{cyl-st2}.  Hence
$\de_{[.acad]}=T^4\de_{[.w_3d]}T^{-4}=T^4(T^{-7}\de_{[.d]}T^7)T^{-4}
=T^{-3}\de_{[.d]}T^3$.  Next, $[.ac]=[.acab]\sqcup[.acac]\sqcup[.acad]$ due
to Lemma \ref{cyl-st2}.  By Lemma \ref{psi-gen2},
\begin{eqnarray*}
\de_{[.ac]}&=&\de_{[.acab]}\de_{[.acad]}\de_{[.acac]}
=(T^{-3}\de_{[.b]}T^3)(T^{-3}\de_{[.d]}T^3)
(\de_{[.acacac]}T^2\de_{[.acacac]}T^{-2})\\
&=&T^{-3}\de_{[.b]}\de_{[.d]}T^3\de_{[.acacac]}T^2\de_{[.acacac]}T^{-2}.
\end{eqnarray*}
Finally, $[.c]=[a.c]=T([.ac])$ so that
$$
\de_{[.c]}=T\de_{[.ac]}T^{-1}=T^{-2}\de_{[.b]}\de_{[.d]}T^3\de_{[.acacac]}
T^2\de_{[.acacac]}T^{-3}.
$$
Since $[.a]=T^{-1}([a.])$ and $[a.]=[.b]\sqcup[.c]\sqcup[.d]$, it follows
from Lemma \ref{psi-gen2} that $\de_{[a.]}=\de_{[.b]}\de_{[.d]}\de_{[.c]}$
and then from Lemma \ref{psi-gen1} that
$$
\de_{[.a]}=T^{-1}\de_{[a.]}T=T^{-1}\de_{[.b]}\de_{[.d]}\de_{[.c]}T
=T^{-1}\de_{[.b]}\de_{[.d]}T^{-2}\de_{[.b]}\de_{[.d]}T^3\de_{[.acacac]}
T^2\de_{[.acacac]}T^{-2}.
$$

\begin{lemma}\label{psi-de}
Any given transformation of the form $\de_U$ is contained in the group
$G_n$ for $n$ large enough.
\end{lemma}

\begin{proof}
If $U$ is an empty set, then $\de_U$ is the identity map.  Now suppose
$U\subset\Om$ is a nonempty clopen set.  Then there exists $n_0\ge1$ such
that for any $n\ge n_0$ the set $U$ can be represented as a union
$U=C_1\cup C_2\cup\dots\cup C_s$, where each $C_i$ is of the form $[u.w]$
for some words $u,w$ of length $2^{n-1}$.  We can assume that the cylinders
$C_1,C_2,\dots,C_s$ are nonempty and distinct.  Then $U=C_1\sqcup C_2\sqcup
\dots\sqcup C_s$.  If $\de_U$ is well defined, then
$\de_U=\de_{C_1}\de_{C_2}\dots\de_{C_s}$ due to Lemma \ref{psi-gen2}.
Since each $C_i$ is a nonempty cylinder of dimension $2^n$, Lemma
\ref{cyl2} implies that $C_i=T^N([.w_nl])$ for some $l\in\{b,c,d\}$ and
$N\in\bZ$.  Then $\de_{C_i}=T^N\de_{[.w_nl]}T^{-N}$ due to Lemma
\ref{psi-gen1}.  In particular, each $\de_{C_i}$ belongs to the group
$G_n$.  It follows that $\de_U\in G_n$ as well.
\end{proof}

For any $n\ge3$ let $H_n$ be the subgroup of $[[T]]$ generated by
$\tau_{[.w_nb]}$, $\tau_{[.w_nc]}$, $\tau_{[.w_nd]}$, and $T$.  The
restriction $n\ge3$ is necessary since $\tau_{[.ac]}$ and $\tau_{[.acac]}$
are not defined.  If $n\ge3$ then the cylinders $[.w_nb]$, $[.w_nc]$, and 
$[.w_nd]$ are contained in $[.acab]$.  Since sets $[.acab]$, 
$T([.acab])=[a.cab]$, and $T^2([.acab])=[ac.ab]$ are disjoint from one 
another, the transformation $\tau_{[.acab]}$ is well defined and so are the 
generators of the group $H_n$.

\begin{lemma}\label{psi-de-tau}
$\de_{[.w_{n+1}l_n]}\in H_{n+2}$ for all $n\ge1$.
\end{lemma}

\begin{proof}
By Lemma \ref{cyl-st3}, $[.w_{n+1}l_n]=T^{N_n}(U_n)\sqcup T^{M_n}(U_n)$,
where $U_n=[.w_{n+2}l_n]$, $N_n=2^{n+1}$, and $M_n=3\cdot2^n$.  Note that
$M_n-N_n\ge2$.  It follows from Lemmas \ref{psi-gen1} and \ref{psi-gen2}
that $\de_{[.w_{n+1}l_n]}=\Psi_{U_n,N_n,N_n+1}\Psi_{U_n,M_n,M_n+1}$.  Since
$\Psi_{U_n,N,N+1}$ is an involution for all $N\in\bZ$, we obtain
$$
\de_{[.w_{n+1}l_n]}=\Psi_{U_n,N_n,N_n+1}\Psi^2_{U_n,N_n+1,N_n+2}
\Psi^2_{U_n,N_n+2,N_n+3}\dots\Psi^2_{U_n,M_n-1,M_n}\Psi_{U_n,M_n,M_n+1}
$$ $$
=(\Psi_{U_n,N_n,N_n+1}\Psi_{U_n,N_n+1,N_n+2})
(\Psi_{U_n,N_n+1,N_n+2}\Psi_{U_n,N_n+2,N_n+3})
\dots(\Psi_{U_n,M_n-1,M_n}\Psi_{U_n,M_n,M_n+1}).
$$
Since $\tau_{U_n}$ is well defined, it follows from Lemma \ref{psi-gen1} 
that $\Psi_{U_n,N,N+2}$ is well defined for any $N\in\bZ$.  By Lemma 
\ref{psi-gen3}, $\Psi_{U_n,N,N+2}=\Psi_{U_n,N,N+1}\Psi_{U_n,N+1,N+2}$ for 
all $N\in\bZ$.  Therefore
$$
\de_{[.w_{n+1}l_n]}=\Psi_{U_n,N_n,N_n+2}\Psi_{U_n,N_n+1,N_n+3}\dots
\Psi_{U_n,M_n-2,M_n}.
$$
By Lemma \ref{psi-gen1}, $\Psi_{U_n,N,N+2}=T^N\tau_{U_n}T^{-N}$ for all 
$N\in\bZ$.  It follows by induction that
$$
\Psi_{U_n,N,N+2}\Psi_{U_n,N+1,N+3}\dots\Psi_{U_n,N+K-1,N+K+1}
=T^N(\tau_{U_n}T)^K\, T^{-N-K}
$$
for all $N\in\bZ$ and $K\ge1$.  In the case $N=N_n$, $K=M_n-N_n-1$, we 
obtain that
$$
\de_{[.w_{n+1}l_n]}=T^{2^{n+1}}(\tau_{[.w_{n+2}l_n]}T)^{2^n-1}\,
T^{1-3\cdot2^n},
$$
which belongs to the group $H_{n+2}$.
\end{proof}

\begin{lemma}\label{psi-tau-all}
$H_n=H_4$ for all $n>4$.
\end{lemma}

\begin{proof}
Let us fix an arbitrary $n\ge4$.  First we are going to show that 
$\tau_{[.w_{n+1}l]}\in H_n$ for each $l\in\{b,c,d\}$ (so that 
$H_{n+1}\subset H_n$).  Let $U=[w_nl_n.]$ and $V_l=[.w_nl]$.  By Lemma 
\ref{seq1}, $w_n$ has length $2^n-1$ and $w_{n+1}=w_nl_nw_n$.  It follows 
that $U=T^{2^n}([.w_nl_n])$ and $U\cap V_l=[w_nl_n.w_nl] 
=T^{2^n}([.w_{n+1}l])$.  By Lemma \ref{cyl-st1}, the cylinder $[.w_nl_n]$ 
is the union of $[.w_{n+1}b]$, $[.w_{n+1}c]$, and $[.w_{n+1}d]$.  Therefore 
$U$ is the union of $T^{2^n}([.w_{n+1}b])=[w_nl_n.w_nb]$, 
$T^{2^n}([.w_{n+1}c])=[w_nl_n.w_nc]$, and 
$T^{2^n}([.w_{n+1}d])=[w_nl_n.w_nd]$.  As a consequence, the cylinder $U$ 
is contained in $[.w_n]$.  Clearly, $V_l\subset[.w_n]$ as well.

By Lemma \ref{cyl1}, the cylinder $[.w_n]$ is disjoint from $T^N([.w_n])$ 
for $1\le N<2^{n-1}$.  In particular, it is disjoint from $T([.w_n])$, 
$T^2([.w_n])$, $T^3([.w_n])$, and $T^4([.w_n])$.  Since $U$ and $V_l$ are 
subsets of $[.w_n]$, the cylinder $U$ is disjoint from $T(V_l)$, 
$T^2(V_l)$, $T^3(V_l)$, and $T^4(V_l)$.  Then it follows from Lemma 
\ref{psi-gen4} that
$$
\Psi_{V_l,0,2}\Psi^{-1}_{U,-2,0}\Psi^{-1}_{V_l,0,2}\Psi_{U,-2,0}
=\Psi_{U\cap V_l,-1,1}.
$$
By Lemma \ref{psi-gen1}, $\Psi_{U\cap V_l,-1,1}=T^{2^n-1}\tau_{[.w_{n+1}l]} 
T^{1-2^n}$ and $\Psi_{U,-2,0}=T^{2^n-2}\tau_{[.w_nl_n]}T^{2-2^n}$.  
Besides, $\Psi_{V_l,0,2}=\tau_{[.w_nl]}$.  Hence
\begin{eqnarray*}
\tau_{[.w_{n+1}l]}&=& T^{1-2^n}\Psi_{U\cap V_l,-1,1}\,T^{2^n-1}
=T^{1-2^n}\Psi_{V_l,0,2}\Psi^{-1}_{U,-2,0}\Psi^{-1}_{V_l,0,2}\Psi_{U,-2,0} 
\,T^{2^n-1}\\
&=& T^{1-2^n} \tau_{[.w_nl]} (T^{2^n-2}\tau_{[.w_nl_n]}T^{2-2^n})^{-1} 
\tau^{-1}_{[.w_nl]} (T^{2^n-2}\tau_{[.w_nl_n]}T^{2-2^n}) T^{2^n-1}\\
&=& T^{1-2^n} \tau_{[.w_nl]} T^{2^n-2}\tau^{-1}_{[.w_nl_n]}T^{2-2^n} 
\tau^{-1}_{[.w_nl]} T^{2^n-2}\tau_{[.w_nl_n]}T,
\end{eqnarray*}
which is in the group $H_n$.

Next we are going to show that $\tau_{[.w_{n-1}l]}\in H_n$ for each 
$l\in\{b,c,d\}$ (so that $H_{n-1}\subset H_n$).  Note that exactly one of
the letters $l_{n-2}$, $l_{n-1}$, and $l_n$ coincides with $l$.  In view of
Lemmas \ref{cyl-st1}, \ref{cyl-st2}, and \ref{cyl-st3}, the cylinder
$[.w_{n-1}l]$ is a disjoint union of one (if $l=l_n$), two (if
$l=l_{n-2}$), or three (if $l=l_{n-1}$) sets of the form $T^N([.w_nl'])$,
where $l'\in\{b,c,d\}$ and $N\in\bZ$.  By Lemma \ref{psi-gen1},
$\tau_{T^N([.w_nl'])}=T^N\tau_{[.w_nl']}T^{-N}$, which belongs to $H_n$.
Then it follows from Lemma \ref{psi-gen2} that $\tau_{[.w_{n-1}l]}$ is a
product of at most three elements of the group $H_n$.  Hence
$\tau_{[.w_{n-1}l]}\in H_n$ as well.

We have shown that $H_{n+1}\subset H_n$ and $H_{n-1}\subset H_n$ for all 
$n\ge4$.  As a consequence, $H_{n+1}=H_n$ for $n\ge4$.  It follows by 
induction that $H_n=H_4$ for $n\ge4$.
\end{proof}

\begin{lemma}\label{psi-de-all}
$G_n=G_3$ for all $n>3$.
\end{lemma}

\begin{proof}
First we are going to show that $\tau_{[.w_4l]}\in G_3$ for each
$l\in\{b,c,d\}$ (so that the group $H_4$ is a subgroup of $G_3$).  Let
$U=[acabacad.]$ and $V_l=[.acabacal]$.  Since $w_3=acabaca$ and $l_3=d$, we
have $U=T^8([.w_3l_3])$, $V_l=[.w_3l]$, and $U\cap V_l=[w_3l_3.w_3l]=
T^8([.w_3l_3w_3l])=T^8([.w_4l])$.  By Lemma \ref{seq-st1}, any occurrence
of $w_3l_3$ in $\xi$ is immediately followed by $w_3b$, $w_3c$, or $w_3d$. 
As a consequence, $U=[w_3l_3.]$ is contained in $[.w_3]$.  Clearly,
$V_l\subset[.w_3]$ as well.  Observe that the cylinder $[.w_3]=[.acabaca]$
is disjoint from $T([.w_3])=[a.cabaca]$ and $T^2([.w_3])=[ac.abaca]$.
Since $U$ and $V_l$ are subsets of $[.w_3]$, the cylinder $U$ is disjoint
from $T(V_l)$ and $T^2(V_l)$.  Then it follows from Lemma \ref{psi-gen4}
that
$$
\Psi_{V_l,0,1}\Psi^{-1}_{U,-1,0}\Psi^{-1}_{V_l,0,1}\Psi_{U,-1,0}
=\Psi_{U\cap V_l,-1,1}.
$$
By Lemma \ref{psi-gen1}, $\Psi_{U\cap V_l,-1,1}=T^7\tau_{[.w_4l]}T^{-7}$
and $\Psi_{U,-1,0}=T^7\de_{[.w_3l_3]}T^{-7}$.  Besides, $\Psi_{V_l,0,1}=
\de_{[.w_3l]}$.  Hence
\begin{eqnarray*}
\tau_{[.w_4l]}&=& T^{-7}\Psi_{U\cap V_l,-1,1}\,T^7
=T^{-7}\Psi_{V_l,0,1}\Psi^{-1}_{U,-1,0}\Psi^{-1}_{V_l,0,1}\Psi_{U,-1,0}\,
T^7\\
&=& T^{-7} \de_{[.w_3l]} (T^7\de_{[.w_3l_3]}T^{-7})^{-1} 
\de^{-1}_{[.w_3l]} (T^7\de_{[.w_3l_3]}T^{-7}) T^7\\
&=& T^{-7} \de_{[.w_3l]} T^7\de_{[.w_3l_3]}T^{-7} 
\de_{[.w_3l]} T^7\de_{[.w_3l_3]},
\end{eqnarray*}
which is in the group $G_3$.

Next we derive three formulas.  By Lemma \ref{cyl-st1}, $[.w_nl_n]=
[.w_{n+1}b]\sqcup[.w_{n+1}c]\sqcup[.w_{n+1}d]$ for all $n\ge1$.  Then Lemma
\ref{psi-gen2} implies that $\de_{[.w_nl_n]}=\de_{[.w_{n+1}b]}
\de_{[.w_{n+1}c]}\de_{[.w_{n+1}d]}$ for $n\ge1$.  By Lemma \ref{cyl-st2},
$[.w_nl_{n+1}]=T^{2^n}([.w_{n+1}l_{n+1}])$ for all $n\ge1$.  Then Lemma
\ref{psi-gen1} implies that $\de_{[.w_nl_{n+1}]}=
T^{2^n}\de_{[.w_{n+1}l_{n+1}]}T^{-2^n}$ for $n\ge1$.  By Lemma
\ref{cyl-st3}, $[.w_nl_{n-1}]=T^{2^n}([.w_{n+1}l_{n-1}])\sqcup
T^{3\cdot2^{n-1}}([.w_{n+1}l_{n-1}])$ for all $n\ge2$.  Then Lemmas
\ref{psi-gen1} and \ref{psi-gen2} imply that
\begin{eqnarray*}
\de_{[.w_nl_{n-1}]}&=&(T^{2^n}\de_{[.w_{n+1}l_{n-1}]}T^{-2^n})
(T^{3\cdot2^{n-1}}\de_{[.w_{n+1}l_{n-1}]}T^{-3\cdot2^{n-1}})\\
&=&T^{2^n}\de_{[.w_{n+1}l_{n-1}]}T^{2^{n-1}}\de_{[.w_{n+1}l_{n-1}]}
T^{-3\cdot2^{n-1}}
\end{eqnarray*}
for $n\ge2$.  Note that for any $n\ge2$ the triple $l_{n-1},l_n,l_{n+1}$ is
a permutation of the triple $b,c,d$.  Therefore the above three formulas
imply that transformations $\de_{[.w_nb]}$, $\de_{[.w_nc]}$, and
$\de_{[.w_nd]}$ belong to the group $G_{n+1}$.  Hence $G_n\subset G_{n+1}$
for all $n\ge2$.

Next we are going to show that $G_{n+1}\subset G_n$ for all $n\ge3$.  By
the above, $H_4\subset G_3$.  In view of Lemmas \ref{psi-de-tau} and
\ref{psi-tau-all}, the group $G_3$ contains $\de_{[.w_{n+1}l_n]}$ for all
$n\ge2$.  Since $G_n\subset G_{n+1}$ for $n\ge2$, it follows by induction
that $G_3\subset G_n$ for all $n\ge3$.  As a consequence,
$\de_{[.w_{n+1}l_n]}\in G_n$ for $n\ge3$.  Besides, for any $n\ge1$ we have
$\de_{[.w_{n+1}l_{n+1}]}=T^{-2^n}\de_{[.w_nl_{n+1}]}T^{2^n}$, which belongs
to $G_n$.  Therefore for any $n\ge3$ the group $G_n$ contains two of the
three transformations $\de_{[.w_{n+1}b]}$, $\de_{[.w_{n+1}c]}$, and
$\de_{[.w_{n+1}d]}$.  Since the product of all three is $\de_{[.w_{n+1}b]}
\de_{[.w_{n+1}c]}\de_{[.w_{n+1}d]}=\de_{[.w_nl_n]}\in G_n$, the remaining
one of the three is in $G_n$ as well.  Hence $G_n$ contains all generators
of the group $G_{n+1}$ so that $G_{n+1}\subset G_n$.

We have shown that $G_n\subset G_{n+1}$ for $n\ge2$ and $G_{n+1}\subset
G_n$ for $n\ge3$.  As a consequence, $G_{n+1}=G_n$ for $n\ge3$.  It follows
by induction that $G_n=G_3$ for all $n\ge3$.
\end{proof}

\begin{proofof}{Theorem \ref{main1}}
According to Theorem \ref{full3}, the topological full group $[[T]]$ is 
generated by $T$ and all transformations of the form $\de_U$, where 
$U\subset\Om$ is a clopen set.  By Lemma \ref{psi-de}, each $\de_U$ is 
contained in the group $G_n$ (generated by $\de_{[.w_nb]}$, 
$\de_{[.w_nc]}$, $\de_{[.w_nd]}$, and $T$) for $n$ large enough.  Then it 
follows from Lemma \ref{psi-de-all} that each $\de_U$ is contained in the 
group $G_3$.  We conclude that $G_3=[[T]]$.  By Lemma \ref{psi-de-3}, the 
group $G_3$ coincides with the group generated by $T$, $\de_{[.b]}$, 
$\de_{[.d]}$, and $\de_{[.acacac]}$.
\end{proofof}

\bigskip

{\sc
\begin{raggedright}
Department of Mathematics\\
Texas A\&M University\\
College Station, TX 77843--3368
\end{raggedright}
}


\begin{thebibliography}{GPS}

\bibitem[GPS]{GPS}
T. Giordano, I. F. Putnam, and C. F. Skau, Full groups of Cantor minimal 
systems. \emph{Israel J. Math.} {\bf 111} (1999), no. 1, 285-–320.

\bibitem[Lys]{Lys}
I.~G.~Lysenok, A system of defining relations for a Grigorchuk group.
\emph{Math. Notes} {\bf 38} (1985), 784--792 [translated from \emph{Mat. 
Zametki} {\bf 38} (1985), no. 4, 503--516].

\bibitem[M-B]{M-B}
N. Matte Bon, Topological full groups of minimal subshifts with subgroups 
of intermediate growth. \emph{J. Modern Dynamics} {\bf 9} (2015), 67-–80.

\bibitem[Mat]{Mat}
H. Matui, Some remarks on topological full groups of Cantor minimal 
systems. \emph{Internat. J. Math.} {\bf 17} (2006), no. 2, 231-–251.

\bibitem[Vor]{Vor}
Ya. Vorobets, On a substitution subshift related to the Grigorchuk group.
\emph{Proc. Steklov Inst. Math.} {\bf 271} (2010), 306--321.
\end{thebibliography}
\end{document}